\documentclass[oneside]{amsart} 
\usepackage[latin1]{inputenc}
\usepackage{amsmath}
\usepackage[inner=2.5cm,outer=25mm, top=25mm, bottom=25mm]{geometry}
\usepackage[latin1]{inputenc}
\usepackage{fancyhdr}
\usepackage{hyperref}
\usepackage{stmaryrd}
\usepackage{tikz}
\usepackage{euscript}
\usepackage{color}
\usepackage{pgfplots}
\usetikzlibrary{calc}       		
\usepackage[scaled=0.9]{helvet}		
\usepackage{courier}			
\usepackage{bbm}
\usepackage[latin1]{inputenc}
\usepackage{listings}
\usepackage{graphicx}
\usepackage{amsmath}
\usepackage{amssymb}
\usepackage{bbm}
\newtheorem{theorem}{Theorem}

\newtheorem{corollary}[theorem]{Corollary}
\newtheorem{definition}[theorem]{Definition}
\newtheorem{example}[theorem]{Example}

\newtheorem{lemma}[theorem]{Lemma}

\newtheorem{proposition}[theorem]{Proposition}
\newtheorem{remark}[theorem]{Remark}

\usepackage{listings}
\usepackage [all,arc,curve,color, arrow, matrix,frame]{xy}
\usepackage[english]{babel}
\usepackage{tikz}

\newcommand{\Z}{{\mathbb Z}}
\newcommand{\R}{{\mathbb R}}
\newcommand{\C}{{\mathbb C}}
\newcommand{\N}{{\mathbb N}}

\newcommand{\Q}{{\mathbb Q}}

\newcommand{\Real}{{\mbox{Re}}}

\newcommand{\spann}{{\mbox{span}}}
\newcommand{\ran}{{\mbox{ran}}}
\newcommand{\Imag}{{\mbox{Im}}}
\newcommand{\idty}{{\mathbbm{1}}}
\newcommand{\sgn}{{\mbox{sgn}}}
\numberwithin{equation}{section}
\numberwithin{theorem}{section} 
\numberwithin{footnote}{section}
\begin{document}

\title[Closed extensions]{The closed extensions of a closed operator}

\author[C. Fischbacher]{Christoph Fischbacher$^1$}
\address{$^1$ Department of Mathematics, University of Alabama at Birmingham, Birmingham, AL 35294, USA}
\email{cfischb@uab.edu}
\numberwithin{equation}{section}
\numberwithin{theorem}{section} 
\numberwithin{footnote}{section}
\maketitle
\newcommand{\chr}[1]{{\color{red}{\it Remark by Christoph:} #1 {\it End of Remark.}}}
%
\begin{abstract} Given a densely defined and closed operator $A$ acting on a complex Hilbert space $\mathcal{H}$, we establish a one-to-one correspondence between its closed extensions and subspaces $\mathfrak{M}\subset\mathcal{D}(A^*)$, that are closed with respect to the graph norm of $A^*$ and satisfy certain conditions. In particular, this will allow us to characterize all densely defined and closed restrictions of $A^*$. After this, we will express our results using the language of Gel'fand triples generalizing the well-known results for the selfadjoint case.
 
As applications we construct: (i) a sequence of densely defined operators that converge in the generalized sense to a non-densely defined operator, (ii) a non-closable extension of a symmetric operator and (iii) selfadjoint extensions of Laplacians with a generalized boundary condition.
\end{abstract}

\section{Introduction}
Since the rigorous study of the theory of operators on Hilbert spaces, their extension theory has always played an important role. In particular, the problem of determining the selfadjoint/maximally sectorial/maximally dissipative extensions of a given symmetric/sectorial/dissipative has been the subject of extensive study over the last decades and it would be impossible to give a complete presentation at this point. For an overview over this field, we thus refer the interested reader to the surveys \cite{Arli, AT2009} and all the references therein.

In this paper, we are going to treat the problem of describing all closed extensions of a given closed and densely defined operator $A$. By taking adjoints, this also leads to a complete description of the densely defined and closed restrictions of $A^*$. To the best of our knowledge, this type of problem has not attracted too much attention so far. At this point, we mention G.\ Grubb's results in \cite{Grubb68}, where closed extensions of so called dual pairs are described with the help of closed auxiliary operators and a more recent work by Z.\ Sebesty\' en and J.\ Stochel \cite{SS2009} in which -- among other results -- densely defined restrictions $A'$ of a given closed and densely defined operator $A$ such that $\dim(\mathcal{D}(A)/\mathcal{D}(A'))=1$ are considered.

We believe that the abstract results of this paper will be of interest to the reader as it provides a self-contained and complete treatment of this problem using only very fundamental classical results and generalizes 
previous results on densely defined restrictions of selfadjoint operators that can be considered as folklore.

We will proceed as follows:

In Section \ref{sec:theory}, we will derive our main result (Thm.\ \ref{thm:corresp}) where we show a one-to-one correspondence between all closed extensions of a given closed and densely defined operator $A$ and subspaces $\mathfrak{M}\subset\mathcal{D}(A^*)$ that are closed with respect to the graph norm of $A^*$ and that satisfy Condition \eqref{eq:precloscon}, which we will introduce below. Note that for our results we do not need to make any further assumptions on $A$ like its resolvent being non-empty. By taking adjoints, we will also derive a description of all densely defined and closed restrictions of $A^*$ (Corollary \ref{coro:denserestrictions}).

In Section \ref{sec:gelfand}, we will reformulate our results on densely defined an closed restrictions of a given densely defined and closed operator in the more natural language of Gel'fand triples (Corollary \ref{coro:gelfanddenserestr}). In particular, this will allow us to generalize results that are well-known for the selfadjoint case (e.g.\ \cite{AT2003}). 

After this, we will discuss the suitable notion of convergence between the so constructed closed extensions and densely defined restrictions in Section \ref{sec:convergence}. To this end, we will recall Kato's notion of generalized convergence for operators on Hilbert spaces.
As an application, we construct a sequence of densely defined and closed symmetric operators of which each domain is described by a ``Riemann-sum"-condition that eventually converges in the generalized sense to a non-densely defined closed Hermitian operator whose domain is described by an integral condition (Example \ref{ex:converge}).

In Section \ref{sec:clonoclo}, we use a side-result obtained from the construction in Section \ref{sec:theory} in order to give two examples of extensions of the selfadjoint momentum operator on the real line, both with infinite codimension. However, only one of the two extensions will be closable while the other one fails to be. 

Finally, in Section \ref{sec:restrictions}, we start out with a selfadjoint reference operator $S$ and use our previously obtained result to construct selfadjoint operators $C_{\phi,\vartheta}$ with the property that $(A+i)^{-1}-(C_{\phi,\vartheta}+i)^{-1}$ has rank one. Here, $\phi$ is a Hilbert space valued parameter and $\vartheta\in(-\pi,\pi]$. As an application, we determine the selfadjoint extensions of Laplacians with a generalized boundary condition (Example \ref{ex:laplacian}). As questions of singular perturbations of selfadjoint operators have been investigated  by numerous authors, \cite{AKN2007, AKN2008, HK2009, KS95, KuraKuro2004, KuroNaga2001, KuzhelZnojil2017, Posilicano} -- just to name a few -- we do not claim that this section contains any new results. It should rather be viewed as another application of the results obtained in the preceding sections.

Let us now fix some notations and conventions: 

When speaking of a Hilbert space $\mathcal{H}$, we always assume $\mathcal{H}$ to be complex. Moreover, the scalar product $\langle\cdot,\cdot\rangle$ on $\mathcal{H}$ is supposed to be antilinear in the first component and linear in the second component.

 Given an operator $A$ on $\mathcal{H}$ with domain $\mathcal{D}(A)$, we denote its graph as $\Gamma(A)$, i.e.\
\begin{equation*}
\Gamma(A)=\{(\psi,A\psi)\in\mathcal{H}\oplus\mathcal{H}:\psi\in\mathcal{D}(A)\}\:.
\end{equation*}
The graph norm of $A$ -- denoted by $\|\cdot\|_{\Gamma(A)}$ -- is given by
\begin{equation*}
\|f\|_{\Gamma(A)}^2=\|f\|^2+\|Af\|^2
\end{equation*}
for any $f\in\mathcal{D}(A)$. The closure of a subspace $\mathcal{K}\subset\mathcal{H}$ with respect to a suitable norm $\|\cdot\|_*$ is denoted by $\overline{\mathcal{K}}^{\|\cdot\|_{*}}$.
Finally, for any closed subspace $\mathcal{M}\subset\mathcal{H}$, let $P(\mathcal{M})$ denote the orthogonal projection onto $\mathcal{M}$.
\section{Main result} \label{sec:theory} 
Given a closed and densely defined operator $A$ on a Hilbert space $\mathcal{H}$, let us use subspaces $\mathfrak{M}\subset\mathcal{D}(A^*)$ in order to parametrize extensions $A_{\mathfrak{M}}$ of $A$:
\begin{definition} \label{def:extension} \normalfont
Let $A$ be a densely defined and closed operator on a complex Hilbert space $\mathcal{H}$. Moreover, let $\mathfrak{M}\subset\mathcal{D}(A^*)$ be a linear space such that
\begin{equation} \label{eq:precloscon}
\ker(A^*)\cap\mathfrak{M}=\{0\}\quad\text{and}\quad\{A^*\phi:\phi\in\mathfrak{M}\}\cap\mathcal{D}(A)=\{0\}\:.
\end{equation}
Then, the operator $A_\mathfrak{M}$, which is an extension of $A$, is defined as
\begin{align*}
A_\mathfrak{M}:\qquad\mathcal{D}(A_\mathfrak{M})&=\mathcal{D}(A)\dot{+}\{A^*\phi:\phi\in\mathfrak{M}\}\\
f+A^*\phi&\mapsto Af-\phi.
\end{align*}
\begin{remark}
Condition \eqref{eq:precloscon} guarantees that $A_\mathfrak{M}$ is well-defined.
\end{remark}
\end{definition}
The proofs of our main theorem will rely on this elementary classical result:
\begin{proposition}[{\cite[Satz 2.49 b]{Weid1}}] A densely defined operator $A$ on a Hilbert space $\mathcal{H}$ is closable if and only if its adjoint $A^*$ is densely defined.
\label{prop:weidmann}
\end{proposition}
Now, let us show that if $\mathfrak{M}$ satisfies the additional requirement \eqref{eq:assumptions}, then $A_{\mathfrak{M}}$ is also closable. 
\begin{lemma} Let $\mathfrak{M}\subset\mathcal{D}(A^*)$ be such that
\begin{equation} \label{eq:assumptions}
\ker A^*\cap \overline{\mathfrak{M}}^{\|\cdot\|_{\Gamma(A^*)}}=\{0\}\quad\text{and}\quad\{A^*\phi:\phi\in\overline{\mathfrak{M}}^{\|\cdot\|_{\Gamma(A^*)}}\}\cap\mathcal{D}(A)=\{0\}.
\end{equation}
Then, the operator $A_\mathfrak{M}$ as defined in Definition \ref{def:extension} is closable, where
\begin{equation*}
\overline{A_\mathfrak{M}}=A_{\overline{\mathfrak{M}}^{\|\cdot\|_{\Gamma(A^*)}}}\:.
\end{equation*}
\label{lemma:closable}
\end{lemma}
\begin{proof} 
Firstly, observe that Condition \eqref{eq:assumptions} ensures that the operator $A_{\overline{\mathfrak{M}}^{\|\cdot\|_{\Gamma(A^*)}}}$ is well-defined.

Now, let us show that $\mathcal{D}(A)$ and $\{A^*\phi:\phi\in\mathfrak{M}\}$ are orthogonal with respect to the inner product induced by the graph norm of $A_\mathfrak{M}$. Thus, take any $f\in\mathcal{D}(A)$ and any $\phi\in\mathfrak{M}$ and consider
\begin{equation*}
\langle f,A^*\phi\rangle_{\Gamma(A_{\mathfrak{M}})}=\langle f,A^*\phi\rangle+\langle A_\mathfrak{M}f,A_\mathfrak{M}A^*\phi\rangle=\langle f,A^*\phi\rangle+\langle Af,-\phi\rangle=\langle f,A^*\phi\rangle-\langle f,A^*\phi\rangle=0\:.
\end{equation*}
This implies that
$$\Gamma(A_\mathfrak{M})=\Gamma(A)\oplus'\{(A^*\phi,-\phi):\phi\in\mathfrak{M}\}\:,$$
where $\oplus'$ denotes the orthogonal sum in $\mathcal{H}\oplus\mathcal{H}$.
Closing with respect to the norm of $\mathcal{H}\oplus\mathcal{H}$ therefore yields
$$\overline{\Gamma(A_\mathfrak{M})}=\overline{\Gamma(A)}\oplus'\overline{\{(A^*\phi,-\phi):\phi\in\mathfrak{M}\}}\:.$$
Since $A$ is closed by assumption, we get that $\overline{\Gamma(A)}=\Gamma(A)$.
Let us now show that 
$$ \overline{\{(A^*\phi,-\phi):\phi\in\mathfrak{M}\}}=\{(A^*\phi,-\phi):\phi\in\overline{\mathfrak{M}}^{\|\cdot\|_{\Gamma(A^*)}}\}\:.  $$
We begin by showing the $``\subset"$ inclusion:

Let $(\psi,-\chi)\in\overline{\{(A^*\phi,-\phi):\phi\in\mathfrak{M}\}}$,  which means that there exists a sequence $\{(A^*\phi_n,-\phi_n)\}_{n=1}^\infty$, where $\{\phi_n\}_{n=1}^\infty\subset\mathfrak{M}$, such that 
\begin{equation*}
\|\psi-A^*\phi_n\|^2+\|\chi-\phi_n\|^2\overset{n\rightarrow\infty}{\longrightarrow}0\:,
\end{equation*}
which means in particular that $\phi_n\rightarrow\chi$ and $A^*\phi_n\rightarrow\psi$. Since $A^*$ is closed, this implies that $\chi\in\mathcal{D}(A^*)$ and $\psi=A^*\chi$.
Hence, any element of $\overline{\{(A^*\phi,-\phi):\phi\in\mathfrak{M}\}}$ is actually of the form $(A^*\chi,-\chi)$ where $\chi\in\mathcal{D}(A^*)$. Furthermore, there exists a sequence $\{\phi_n\}_{n=1}^\infty\subset\mathfrak{M}$ such that
\begin{equation*}
\|A^*(\chi-\phi_n)\|^2+\|\chi-\phi_n\|^2=\|\chi-\phi_n\|^2_{\Gamma(A^*)}\overset{n\rightarrow\infty}{\longrightarrow}0\:,
\end{equation*}
which means that $\chi\in\overline{\mathfrak{M}}^{\|\cdot\|_{\Gamma(A^*)}}$.

Next, let us show the $``\supset"$ inclusion:

To see this, we need to show that if $\phi\in\overline{\mathfrak{M}}^{\|\cdot\|_{\Gamma(A^*)}}$, this implies that $(A^*\phi,-\phi)\in\overline{\{(A^*\phi,-\phi):\phi\in\mathfrak{M}\}}$. But if $\phi\in\overline{\mathfrak{M}}^{\|\cdot\|_{\Gamma(A^*)}}$, there exists a sequence $\{\phi_n\}_{n=1}^\infty\subset\mathfrak{M}$ such that 
\begin{equation*}
\|\phi-\phi_n\|^2_{\Gamma(A^*)}\overset{n\rightarrow\infty}{\longrightarrow}0
\end{equation*}
and since
$$ \|\phi-\phi_n\|^2_{\Gamma(A^*)}=\|\phi-\phi_n\|^2+\|A^*(\phi-\phi_n)\|^2=\|(A^*\phi,-\phi)-(A^*\phi_n,-\phi_n)\|^2_{\mathcal{H}\oplus\mathcal{H}}\:,$$
this shows that $(A^*\phi,-\phi)\in\overline{\{(A^*\phi,-\phi):\phi\in\mathfrak{M}\}}$. We therefore have shown that
\begin{equation*}
\overline{\Gamma(A_\mathfrak{M})}=\Gamma(A)\oplus ' \{A^*\phi:\phi\in
\overline{\mathfrak{M}}^{\|\cdot\|_{\Gamma(A^*)}}\}\:.
\end{equation*}
Let us finish by arguing that $\overline{\Gamma(A_\mathfrak{M})}$ is the graph of an operator, which means that we need to show that $(0,g)\in\overline{\Gamma(A_\mathfrak{M})}$ implies that $g=0$. But any element of $\overline{\Gamma(A_\mathfrak{M})}$ is of the form $(f+A^*\phi, Af-\phi)$, where $f\in\mathcal{D}(A)$ and $\phi\in\overline{\mathfrak{M}}^{\|\cdot\|_{\Gamma(A^*)}}$. Moreover, by \eqref{eq:assumptions}, we have that $f+A^*\phi=0$ if and only if $f=0$ and $A^*\phi=0$. Since --- again by \eqref{eq:assumptions} --- we have that $A^*\phi=0$ if and only if $\phi=0$, this yields that $(f+A^*\phi,Af-\phi)=(0,Af-\phi)=(0,0)$, which implies that $\overline{\Gamma(A_\mathfrak{M})}$ is the graph of an operator which therefore must be the closure $\overline{A_\mathfrak{M}}$ of $A_\mathfrak{M}$. In particular, this implies that $A_\mathfrak{M}$ is closable. Moreover, $\overline{A_\mathfrak{M}}$ is equal to $A_{\overline{\mathfrak{M}}^{\|\cdot\|_{\Gamma(A^*)}}}$. This shows the lemma.
\end{proof}
\begin{remark} For the case that $\mathfrak{M}$ is finite-dimensional, Conditions \eqref{eq:precloscon} and \eqref{eq:assumptions} coincide. Later, we will give an example of an infinite-dimensional subspace $\mathfrak{M}$ which satisfies Condition \eqref{eq:precloscon}, but fails to satisfy Condition \eqref{eq:assumptions} (cf.\ Section \ref{sec:clonoclo}). In this case, the operator $A_\mathfrak{M}$ will be non-closable (cf.\ also Corollary \ref{coro:nonclosable}.)
\end{remark}
The following lemma provides an alternative characterization of $A_\mathfrak{M}$.
\begin{lemma} Let $\mathfrak{M}\subset\mathcal{D}(A^*)$ be a linear space satisfying Condition \eqref{eq:precloscon} and let the operator $B_\mathfrak{M}$ be given by:
\begin{align*}
B_\mathfrak{M}:\qquad\mathcal{D}(B_\mathfrak{M})&=\{f\in\mathcal{H}: \exists\phi\in\mathfrak{M}\:\:\text{such that}\:\:f-A^*\phi\in\mathcal{D}(A)\}\\
B_\mathfrak{M}f&=A(f-A^*\phi)-\phi\:.
\end{align*}
Then, $B_\mathfrak{M}=A_\mathfrak{M}$.
\end{lemma}
\begin{proof}
Since $\mathfrak{M}$ satisfies Condition \eqref{eq:precloscon}, it is obvious that $B_\mathfrak{M}$ is well--defined.\\
``$A_\mathfrak{M}\subset B_\mathfrak{M}$": Any $f_0+A^*\phi$ with $f_0\in\mathcal{D}(A)$ and $\phi\in\mathfrak{M}$ is also in $\mathcal{D}(B_\mathfrak{M})$ as $(f_0+A^*\phi-A^*\phi)\in\mathcal{D}(A)$. Now, consider
\begin{equation*}
B_\mathfrak{M}(f_0+A^*\phi)=A(f_0+A^*\phi-A^*\phi)-\phi=Af_0-\phi=A_\mathfrak{M}(f_0+A^*\phi)\:,
\end{equation*}
which shows the first inclusion.\\
``$A_\mathfrak{M}\supset B_\mathfrak{M}$": Observe that for any $f\in\mathcal{D}(B_\mathfrak{M})$, there exists a $\phi\in\mathfrak{M}$ such that $f$ can be written as $f=(f-A^*\phi)+A^*\phi$, where $(f-A^*\phi)\in\mathcal{D}(A)$. This implies that $f\in\mathcal{D}(B_\mathfrak{M})$ as well. To finish the proof, consider
\begin{equation*}
A_\mathfrak{M}f=A_\mathfrak{M}(f-A^*\phi+A^*\phi)=A(f-A^*\phi)-\phi=B_\mathfrak{M}f\:.
\end{equation*}
\end{proof}
Next, assuming that $A$ is a densely defined and closed operator on $\mathcal{H}$, let us introduce the restriction $C_\mathfrak{M}(A)$ of $A^*$. We parametrize $C_\mathfrak{M}(A)$ by an orthogonality condition in $\Gamma(A^*)$:
\begin{definition} \normalfont Let $A$ be a closed and densely defined operator on a Hilbert space $\mathcal{H}$. Moreover, let $\mathfrak{M}\subset\mathcal{D}(A^*)$. Then, the operator $C_\mathfrak{M}(A)$ is defined as
\begin{align*}
C_\mathfrak{M}(A):\qquad\mathcal{D}(C_\mathfrak{M}(A))&=\{f\in\mathcal{D}(A^*): \langle f,\phi\rangle+\langle A^*f,A^*\phi\rangle=0\:\:\text{for all}\:\: \phi\in\mathfrak{M}\}\\
C_{\mathfrak{M}}(A)&=A^*\upharpoonright_{\mathcal{D}(C_\mathfrak{M}(A))}\:.
\end{align*} \label{def:restr}
\begin{remark} Even though $C_\mathfrak{M}(A)$ depends on the operator $A$, in most cases we will just write $C_\mathfrak{M}$.
\end{remark}
\begin{remark}
Note that --- unlike for the definition of $A_\mathfrak{M}$ (cf.\ Definition \ref{def:extension}) --- we have not made any additional assumptions on $\mathfrak{M}\subset\mathcal{D}(A^*)$. Moreover, observe that for $\mathfrak{M}\neq\{0\}$, the operator $C_\mathfrak{M}$ is a proper restriction of $A^*$. Indeed, let $0\neq\phi\in\mathfrak{M}$. Since $$\langle\phi,\phi\rangle+\langle A^*\phi,A^*\phi\rangle=\|\phi\|^2+\|A^*\phi\|^2\neq 0\:,$$
this immediately implies that $\phi\notin\mathcal{D}(C_\mathfrak{M})$. 
\end{remark}
\end{definition} 
Next, let us show that if $\mathfrak{M}\subset\mathcal{D}(A^*)$ satisfies Condition \eqref{eq:precloscon}, we then get that $A_{\mathfrak{M}}^*=C_\mathfrak{M}$.
\begin{lemma} Let $A$ be a densely defined and closed operator on a Hilbert space $\mathcal{H}$ and let $\mathfrak{M}\subset\mathcal{D}(A^*)$  satisfy Condition \eqref{eq:precloscon}. Let $A_\mathfrak{M}$ be defined as in Definition \ref{def:extension} and $C_\mathfrak{M}$ be defined as in Definition \ref{def:restr}. Then, $A_\mathfrak{M}^*=C_\mathfrak{M}$. \label{lemma:adjeq}
\end{lemma}
\begin{proof} ``$C_\mathfrak{M}\subset A_\mathfrak{M}^*$": Let $g\in\mathcal{D}(C_\mathfrak{M})$, $f\in\mathcal{D}(A)$ and $\phi\in\mathfrak{M}$ and consider
\begin{align*}
\langle g, A_\mathfrak{M}(f+A^*\phi)\rangle=\langle g, Af-\phi\rangle=\langle A^*g, f+A^*\phi\rangle\:,
\end{align*}
where we have used that $g\in\mathcal{D}(A^*)$ and $-\langle g,\phi\rangle=\langle A^*g,A^*\phi\rangle$. This shows that $g\in\mathcal{D}(A_\mathfrak{M}^*)$ and $A_\mathfrak{M}^*g=A^*g=C_\mathfrak{M}g$.\\
``$C_\mathfrak{M}\supset A_\mathfrak{M}^*$": Let $g\in\mathcal{D}(A_\mathfrak{M}^*)$, which means that there exists a $\widetilde{g}\in\mathcal{H}$ such that
\begin{equation} \label{eq:impl1}
\langle \widetilde{g},f+A^*\phi\rangle=\langle g,A_\mathfrak{M}(f+A^*\phi)\rangle=\langle g, Af-\phi\rangle
\end{equation}
for all $f\in\mathcal{D}(A)$ and all $\phi\in\mathfrak{M}$. This holds in particular for the choice $\phi=0$, from which we get that
\begin{equation*}
\langle \widetilde{g},f\rangle=\langle g,A_\mathfrak{M}f\rangle=\langle g,Af\rangle
\end{equation*}
for all $f\in\mathcal{D}(A)$. This implies that $g\in\mathcal{D}(A^*)$ and that $\widetilde{g}=A^*g$. Now, consider again Equation \eqref{eq:impl1}:
\begin{equation*}
\langle A^*g,f+A^*\phi\rangle=\langle \widetilde{g},f+A^*\phi\rangle=\langle g,Af-\phi\rangle=\langle A^*g,f\rangle-\langle g,\phi\rangle\:,
\end{equation*}
which implies that $$\langle g,\phi\rangle+\langle A^*g,A^*\phi\rangle=0$$
for all $\phi\in\mathfrak{M}$. This shows that $g\in\mathcal{D}(C_\mathfrak{M})$ and $A_\mathfrak{M}^*g=A^*g=C_\mathfrak{M}g$, from which the lemma follows.
\end{proof}

\begin{theorem} Let $\mathfrak{M}\subset\mathcal{D}(A^*)$ be a linear space. Then, the operator $C_\mathfrak{M}$ as given in Definition \ref{def:restr} is a closed restriction of $A^*$. Moreover, it is densely defined if and only if $\mathfrak{M}$ satisfies Condition \eqref{eq:assumptions}.
\label{thm:denselydefined}
\end{theorem}
\begin{proof}
The fact that $C_\mathfrak{M}$ is a restriction of $A^*$ is a trivial consequence of its definition. It is also not hard to see that 
\begin{equation} \label{eq:perpi}
\Gamma(C_\mathfrak{M})=\Gamma(A^*)\cap\{(\phi,A^*\phi):\phi\in\mathfrak{M}\}^{\perp}\:,
\end{equation}
where the orthogonal complement is taken in $\mathcal{H}\oplus\mathcal{H}$. But this implies that $\Gamma(C_\mathfrak{M})$ is closed in $\mathcal{H}\oplus\mathcal{H}$, from which we deduce that $C_\mathfrak{M}$ is a closed operator.
\\
Let us now show that Condition \eqref{eq:assumptions} is necessary for $C_\mathfrak{M}$ to be densely defined. Assume that there exists a $0\neq\phi\in\overline{\mathfrak{M}}^{\|\cdot\|_{\Gamma(A^*)}}$ such that $A^*\phi\in\mathcal{D}(A)$. This would mean that there exists a sequence $\{\phi_n\}_{n=1}^\infty\subset\mathfrak{M}$ such that 
$$\lim_{n\rightarrow\infty}\left(\|\phi_n-\phi\|^2+\|A^*\phi_n-A^*\phi\|^2\right)=0\:.$$
Since for any $n\in\N$ and any $f\in\mathcal{D}(C_\mathfrak{M})$ we have
$$\langle f,\phi_n\rangle+\langle A^*f,A^*\phi_n\rangle=0$$
and 
$$ \langle f,\phi\rangle+\langle A^*f,A^*\phi\rangle=\lim_{n\rightarrow\infty}(\langle f,\phi_n\rangle+\langle A^*f,A^*\phi_n\rangle)=0\:,$$
we obtain the condition $$\langle f, (\idty+AA^*)\phi\rangle=0$$ for all $f\in\mathcal{D}(C_\mathfrak{M})$. This means that $\mathcal{D}(A_\mathfrak{M}^*)\perp\spann\{(\idty+AA^*)\phi\}$, which implies that $C_\mathfrak{M}$ is not densely defined. Note that for $\phi\neq 0$, it cannot happen that $(\idty+AA^*)\phi=0$, since we would get
\begin{equation*}
\phi=-AA^*\phi\quad\Rightarrow\quad \|\phi\|^2=\langle\phi,\phi\rangle=-\langle\phi,AA^*\phi\rangle=-\|A^*\phi\|^2\:,
\end{equation*}
which is impossible.
\\
Let us now show that Condition \eqref{eq:assumptions} is sufficient for $C_\mathfrak{M}$ to be densely defined. This follows from Lemma \ref{lemma:closable}, from which we have that $A_\mathfrak{M}$ is closable. Proposition \ref{prop:weidmann} then implies that $A_\mathfrak{M}^*$ is densely defined, which together with the fact that $A_\mathfrak{M}^*=C_\mathfrak{M}$ by Lemma \ref{lemma:adjeq} finishes the proof.
\end{proof}
\begin{corollary} Let $\mathfrak{M}\subset\mathcal{D}(A^*)$ be a linear space that satisfies \eqref{eq:precloscon} but does not satisfy \eqref{eq:assumptions}. Then, the operator $A_\mathfrak{M}$ as defined in Definition \ref{def:extension} is not closable. \label{coro:nonclosable}
\end{corollary}
\begin{proof} By Lemma \ref{lemma:adjeq}, we have that $A_{\mathfrak{M}}^*=C_\mathfrak{M}$. However, since $\mathfrak{M}$ does not satisfy Condition \eqref{eq:assumptions}, we have by Theorem \ref{thm:denselydefined} that $C_\mathfrak{M}=A_\mathfrak{M}^*$ is not densely defined. Thus, by Proposition \ref{prop:weidmann}, $A_\mathfrak{M}$ is not closable.
\end{proof}
Let us summarize our results with the following theorem, which establishes a one-to-one correspondence between all closed extensions of a given densely defined and closed operator and all subspaces $\mathfrak{M}\subset\mathcal{D}(A^*)$ that are closed with respect to $\|\cdot\|_{\Gamma(A^*)}$ and that satisfy Condition \eqref{eq:assumptions}:
\begin{theorem} \label{thm:posta} Let $A$ be a densely defined and closed operator. Then, there is a one-to-one correspondence between all closed extensions of $A$ and all subspaces $\mathfrak{M}\subset\mathcal{D}(A^*)$ that are closed with respect to the graph norm $\|\cdot\|_{\Gamma(A^*)}$ and that satisfy the conditions given in \eqref{eq:assumptions}. These closed extensions of $A$ are given by
\begin{align} \label{eq:extform}
A_\mathfrak{M}:\qquad\mathcal{D}(A_\mathfrak{M})&=\mathcal{D}(A)\dot{+}\{A^*\phi: \phi\in\mathfrak{M}\}\notag\\
f+A^*\phi&\mapsto Af-\phi\:.
\end{align} \label{thm:corresp}
\end{theorem}
\begin{proof} Let $B$ be any closed extension of $A$. By Proposition \ref{prop:weidmann}, this implies that $B^*$ is densely defined and since $B^*\subset A^*$, this means that $B^*$ is a closed densely defined restriction of $A^*$. Thus,
$$ \Gamma:=\Gamma(A^*)\ominus\Gamma(B^*) $$
is a closed subspace of $\Gamma(A^*)$ and moreover we have $\Gamma(B^*)=\Gamma(A^*)\ominus\Gamma=\Gamma(A^*)\cap\Gamma^{\perp}$. Defining $\mathfrak{M}:=\{\phi\in\mathcal{D}(A^*): (\phi,A^*\phi)\in\Gamma\}$, we then may write
\begin{align}
B^*:\quad\mathcal{D}(B^*)&=\{f\in\mathcal{D}(A^*): \langle f,\phi\rangle+\langle A^*f,A^*\phi\rangle=0\:\:\text{for all}\:\:\phi\in\mathfrak{M}\}\notag\\
B^*&=A^*\upharpoonright_{\mathcal{D}(B^*)}\:.
\label{eq:auxeq}
\end{align}
Moreover, since $\Gamma$ is closed in $\mathcal{H}\oplus\mathcal{H}$, observe that $\mathfrak{M}$ is closed with respect to the graph norm $\|\cdot\|_{\Gamma(A^*)}$, since for any $\phi\in\mathfrak{M}$ we have
\begin{equation*}
\|\phi\|^2_{\Gamma(A^*)}=\|\phi\|^2+\|A^*\phi\|^2=\|(\phi,A^*\phi)\|^2_{\mathcal{H}\oplus\mathcal{H}}\:.
\end{equation*}
Now, \eqref{eq:auxeq} means that $B^*= C_\mathfrak{M}$, where $C_\mathfrak{M}$ is defined as in Definition \ref{def:restr}. By Theorem \ref{thm:denselydefined}, $B^*=C_\mathfrak{M}$ being densely defined implies that $\mathfrak{M}$ satisfies the conditions given in \eqref{eq:assumptions}, which be Lemma \ref{lemma:adjeq} implies that $C_\mathfrak{M}=A_\mathfrak{M}^*$. Also, since $\mathfrak{M}$ is closed with respect to the graph norm $\|\cdot\|_{\Gamma(A^*)}$, we have by Lemma \ref{lemma:closable} that $A_\mathfrak{M}$ is closed. Finally, since $B^*=A_\mathfrak{M}^*$ and $B$ as well as $A_\mathfrak{M}$ are closed, we get that $B\equiv A_\mathfrak{M}$, i.e. any closed extension $B$ of $A$ is of the form $B=A_\mathfrak{M}$, where $\mathfrak{M}$ is a subspace of $\mathcal{D}(A^*)$ that is closed with respect to the graph norm $\|\cdot\|_{\Gamma(A^*)}$ and satisfies the conditions given by \eqref{eq:assumptions}. Finally, let us argue that the mapping $\mathfrak{M}\mapsto A_\mathfrak{M}$ is injective. But $A_\mathfrak{M}=A_{\mathfrak{M}'}$ would imply that $A_\mathfrak{M}^*=A_{\mathfrak{M}'}^*$. However as argued above, we have that 
$$\mathfrak{M}=\{\phi\in\mathcal{D}(A^*):(\phi,A^*\phi)\in\Gamma(A^*)\ominus\Gamma(A_{\mathfrak{M}}^*)\}=\{\phi\in\mathcal{D}(A^*):(\phi,A^*\phi)\in\Gamma(A^*)\ominus\Gamma(A_{\mathfrak{M}'}^*)\}=\mathfrak{M}'\:,$$
showing that $\mathfrak{M}\mapsto A_\mathfrak{M}$ is injective.
This finishes the proof.  
\end{proof}
Likewise, since each closed extension of $A$ is the adjoint of a closed and densely defined restriction of $A^*$, we have also established a one-to-one correspondence between all densely defined and closed restrictions of $A^*$:

\begin{corollary} \label{coro:denserestrictions} Let $A$ be a densely defined and closable operator. Then, there is a one-to-one correspondence between all densely defined and closed restrictions of $A^*$ and all subspaces $\mathfrak{M}\subset\mathcal{D}(A^*)$ that are closed with respect to $\|\cdot\|_{\Gamma(A  ^*)}$ and that satisfy Condition \eqref{eq:assumptions}. These densely defined and closed restrictions are given by the operators $C_\mathfrak{M}$ as defined in Definition \ref{def:restr}.
\end{corollary}
\begin{proof}
By Proposition \ref{prop:weidmann}, since $A$ is densely defined and closable, we know that $A^*$ is densely defined. Now, let $C$ be any closed and densely defined restriction of $A^*$. This immediately implies that $A\subset C^*$ and by Theorem \ref{thm:corresp}, there exists a unique subspace $\mathfrak{M}\subset\mathcal{D}(A^*)$, which is closed with respect to $\|\cdot\|_{\Gamma(^*)}$ and which satisfies Condition \eqref{eq:assumptions} such that $C^*=A_\mathfrak{M}$. Since $C$ is closed, we get $A_{\mathfrak{M}}^*=C^{**}=C$ and Lemma \ref{lemma:adjeq} then implies that $C=C_\mathfrak{M}$, which shows the corollary.
\end{proof}
\section{Gel'fand Triples} \label{sec:gelfand}
The purpose of this section is to add an additional and in some sense more natural point of view to the results we have obtained so far with the help of Gel'fand triples. Our construction is motivated by \cite{CrandallPhillips}, where we adjust a few necessary details as suitable for our later needs. 

To this end, for any densely defined and closed operator $A^*$, let the Hilbert space $\mathcal{H}_{+1}$ be the linear space $\mathcal{D}(A^*)$ equipped with the inner product $\langle f,g\rangle_{+1}:=\langle f,g\rangle+\langle A^*f,A^*g\rangle$, which continuously embeds into $\mathcal{H}$. It can be shown that $\mathcal{D}(A^*)=\mathcal{D}((\idty+AA^*)^{1/2})$ and
\begin{equation} \label{eq:sqrtidty}
\langle f,g\rangle_{+1}=\langle f,g\rangle+\langle A^*f,A^*g\rangle=\langle (\idty+AA^*)^{1/2}f,(\idty+AA^*)^{1/2}g\rangle
\end{equation}
for any $f,g\in\mathcal{H}_{+1}$.

Let us also introduce the Hilbert space $\mathcal{H}_{-1}$ as the closure of $\mathcal{H}$ with respect to the norm $\|\cdot\|_{-1}$, which is induced by the inner product $\langle f,g\rangle_{-1}:=\langle (\idty+AA^*)^{-1/2}f,(\idty+AA^*)^{-1/2}g\rangle$ for any $f,g\in\mathcal{H}$. By a well-known theorem of J.\ von Neumann \cite[Satz 8.22 b]{Weid1}, $AA^*$ is a non-negative selfadjoint operator, which implies in particular that $(\idty+AA^*)^{1/2}$ is boundedly invertible. Moreover note that $\|f\|_{-1}\leq \|f\|$ for any $f\in\mathcal{H}$. Also, note that by virtue of the same theorem, $\mathcal{D}(AA^*)$ is a core of $(\idty+AA^*)^{1/2}$.

The space $\mathcal{H}_{-1}$ can now be identified with $\mathcal{H}_{+1}^*$ in the following sense:

For any bounded linear functional $\ell\in\mathcal{H}_{+1}^*$, there exists by Riesz' representation theorem a unique $\phi\in\mathcal{H}_{+1}=\mathcal{D}((\idty+AA^*)^{1/2})$ such that for any $f\in\mathcal{H}_{+1}$:
\begin{equation*}
\ell(f)=\langle \phi,f\rangle_{+1}=\langle (\idty+AA^*)^{1/2}\phi,(\idty+AA^*)^{1/2}f\rangle\:.
\end{equation*}
Now, since $\mathcal{D}(AA^*)$ is a core of $(\idty+AA^*)^{1/2}$, there exists a sequence $\{\phi_n\}_{n=1}^\infty\subset\mathcal{D}(AA^*)$ such that $\phi_n\rightarrow \phi$ and $(\idty+AA^*)^{1/2}\phi_n\rightarrow(\idty+AA^*)^{1/2}\phi$.
We then get that for any $f\in\mathcal{D}((\idty+AA^*)^{1/2})$:
\begin{align*} 
\ell(f)&=\langle \phi,f\rangle_{+1}=\langle (\idty+AA^*)^{1/2}\phi,(\idty+AA^*)^{1/2}f\rangle=\lim_{n\rightarrow\infty}\langle (\idty+AA^*)^{1/2}\phi_n,(\idty+AA^*)^{1/2}f\rangle\notag\\&=\lim_{n\rightarrow\infty}\langle (\idty+AA^*)\phi_n,f\rangle=\lim_{n\rightarrow\infty}\langle \ell_n,f\rangle\:.
\end{align*}
where we have defined $\ell_n:=(\idty+AA^*)\phi_n$ for any $n\in\N$.
This means that for any bounded linear functional $\ell\in\mathcal{H}_{+1}^*$ there exists a sequence $\{\ell_n\}_{n=1}^\infty$ of elements in $\mathcal{H}$ which is convergent in $\|\cdot\|_{-1}$--norm such that for any $f\in\mathcal{H}_{+1}$ we have $\ell(f)=\lim_{n\rightarrow\infty}\langle \ell_n,f\rangle$
and whose $\|\cdot\|_{-1}$--limit we then identify with $\ell\in\mathcal{H}_{+1}^*$. Conversely, note also that for any sequence $\{w_n\}_{n=1}^\infty\subset\mathcal{H}$ which is convergent in $\|\cdot\|_{-1}$--norm we get for any $f\in\mathcal{H}_{+1}$:
\begin{equation*}
\left|\lim_{n\rightarrow\infty}\langle w_n,f\rangle\right|=\left|\lim_{n\rightarrow\infty}\langle (\idty+AA^*)^{-1/2}w_n,(\idty+AA^*)^{1/2}f\rangle\right|=|\langle \widehat{w},(\idty+AA^*)^{1/2}f\rangle|\leq \|\widehat{w}\|\|f\|_{+1}\:,
\end{equation*}
where $\widehat{w}\in\mathcal{H}$ is the limit of $\{(\idty+AA^*)^{-1/2}w_n\}_{n=1}^\infty$ in $\mathcal{H}$, which exists since $\{w_n\}_{n=1}^\infty$ is convergent in $\mathcal{H}_{-1}$. Hence, any such sequence $\{w_n\}_{n=1}^\infty$ defines a bounded linear functional on $\mathcal{H}_{+1}$ via $f\mapsto \lim_{n\rightarrow\infty}\langle w_n,f\rangle$.
In this sense, we have explicitly constructed the Gel'fand triple: $\mathcal{H}_{+1}\subset\mathcal{H}\subset\mathcal{H}_{-1}$.

In view of this framework, Theorem \ref{thm:denselydefined} can then be reformulated as follows:
\begin{theorem} Let $A$ be a densely defined and closable operator and let $\mathcal{L}\subset\mathcal{H}_{-1}$. Then, the operator $A_\mathcal{L}'$ given by
\begin{equation} \label{eq:al}
A_\mathcal{L}':\quad\mathcal{D}(A_\mathcal{L}')=\{f\in\mathcal{D}(A^*): \forall \ell\in \mathcal{L}: \ell(f)=0\},\qquad A_\mathcal{L}'=A^*\upharpoonright_{\mathcal{D}(A_\mathcal{L}')}
\end{equation}
 is closed. Moreover, it is densely defined if and only if
\begin{equation} \label{eq:-1cond}
\overline{\mathcal{L}}^{\|\cdot\|_{-1}}\cap\mathcal{H}=\{0\}\:.
\end{equation} \label{thm:densefctl}
\end{theorem}
\begin{proof} By Riesz' representation theorem, for any $\ell\in\mathcal{L}$, there exists a unique $\phi_\ell\in\mathcal{D}(A^*)=\mathcal{H}_{+1}$ such that 
\begin{equation*} \label{eq:riesz}
\forall f\in\mathcal{D}(A^*):\quad\ell(f)=\langle \phi_\ell,f\rangle+\langle A^*\phi_\ell,A^*f\rangle\:.
\end{equation*}
Defining the set $$\mathfrak{M}:=\{\phi\in\mathcal{D}(A^*):\exists \ell\in\mathcal{L} \:\mbox{s.t.}\:\forall f\in\mathcal{D}(A^*): \ell(f)=\langle\phi,f\rangle+\langle A^*\phi,A^*f\rangle\}\:,$$ we see that $A_\mathcal{L}'=C_\mathfrak{M}$, where $C_\mathfrak{M}$ was defined in \ref{def:restr}. Theorem \ref{thm:denselydefined} then immediately implies that $A_\mathcal{L}'$ is closed. 

Assume now that there exists a $0\neq\ell\in\overline{\mathcal{L}}^{\|\cdot\|_{-1}}\cap\mathcal{H}$. Then, there exists a sequence $\{\ell_n\}_{n=1}^\infty\subset\mathcal{L}$ such that for any $f\in\mathcal{H}_{+1}$ we have $\ell(f)=\lim_{n\rightarrow\infty}\ell_n(f)$. For any $f\in\mathcal{D}(A_\mathcal{L}')$ this means in particular that
\begin{equation*}
\langle \ell,f\rangle=\ell(f)=\lim_{n\rightarrow\infty}\ell_n(f)=\lim_{n\rightarrow\infty}0=0\:,
\end{equation*}
implying that $\mathcal{D}(A'_{\mathcal{L}})\perp\spann\{\ell\}$, which therefore is not dense.  

Now, assume that $A_\mathcal{L}'=C_\mathfrak{M}$ is not densely defined. By Theorem \ref{thm:denselydefined}, this means that there exists a $0\neq \phi\in\overline{\mathfrak{M}}^{\|\cdot\|_{\Gamma(A^*}}$ such that $A^*\phi\in\mathcal{D}(A)$. Let us now show that the bounded linear functional on $\mathcal{H}_{+1}$ given by $\mathcal{H}\ni(\idty+AA^*)\phi: f\mapsto \langle \phi,f\rangle+\langle A^*\phi,A^*f\rangle=\langle(\idty+ AA^*)\phi,f\rangle$, is an element of $\overline{\mathcal{L}}^{\|\cdot\|_{-1}}$ proving that $\overline{\mathcal{L}}^{\|\cdot\|_{-1}}\cap\mathcal{H}\neq \{0\}$. To this end, let $\{\phi_n\}_{n=1}^\infty\subset \mathfrak{M}$ be such that $\phi_n\rightarrow\phi$ and $A^*\phi_n\rightarrow A^*\phi$. By definition of $\mathfrak{M}$, the linear functionals $\ell_n:\: f\mapsto \langle \phi_n,f\rangle+\langle A^*\phi_n,A^*f\rangle$ are all elements of $\mathcal{L}$. It is now not hard to see that $\ell_n\rightarrow\ell$ with respect to $\|\cdot\|_{-1}$ which thus finishes the proof.
   
\end{proof}
\begin{example} \normalfont Let $\mathcal{H}=L^2(\R^2)$, $s\geq 1$ and the selfadjoint operator $A=A^*$ be given by
\begin{equation*}
A:\quad\mathcal{D}(A)=H^s(\R^2),\qquad f\mapsto|\nabla|^s f\:.
\end{equation*}
Here and in the following $H^s$ denotes the Sobolev space of order $s$.
We can identify $\mathcal{H}_{+1}=H^s(\R^2)$ and $\mathcal{H}_{-1}=H^{-s}(\R^2)$. Let $\{\eta_1,\eta_2,\dots,\eta_n\}$ be a linearly independent set of measurable functions such that for any $i=\{1,2,\dots,n\}$, there is a $q_i\in(1,2)$ such that $\eta_i\in L^{q_i}(\R^2)\setminus L^2(\R^2)$. Moreover, let $\{g_1,g_2,\dots, g_n\}$ be a collection (of not necessarily linear independent) elements of $L^2(\R^2)$. 

By Sobolev embedding (e.g.\ \cite[Thm.\ 8.5 (ii)]{LiebLoss}), the map $f\mapsto \int_{\R^2}\overline{(\eta_i-g_i)}f\:dx$ is a bounded linear functional on $H^s(\R^2)$, since
\begin{equation*}
\left|\int_{\R^2}\overline{(\eta_i-g_i)}f\:dx\right|\leq \|\eta_i\|_{q_i}\|f\|_{\frac{q_i}{q_i-1}}+\|g_i\|_2\|f\|_2\leq C_i\sqrt{\|f\|_2^2+\|\nabla f\|_2^2}=C_i\|f\|_{H^1}\leq C_i\|f\|_{H^s}\:,
\end{equation*}
where the constant $C_i$ only depends on the fixed parameters $q_i, \|\eta_i\|_{q_i}$ and $\|g_i\|_2$. Defining $\mathcal{L}:=\spann\{\eta_1-g_1,\dots, \eta_n-g_n\}$, firstly note that $\mathcal{L}=\overline{\mathcal{L}}^{\|\cdot\|_{-1}}$ since $\mathcal{L}$ is finite-dimensional. Since we made the assumption $\eta_i\notin L^2(\R^2)$, we get that $\overline{\mathcal{L}}^{\|\cdot\|_{-1}}\cap\mathcal{H}=\mathcal{L}\cap\mathcal{H}=\{0\}$. Hence, by Theorem \ref{thm:densefctl}, the set
\begin{equation*}
\mathfrak{D}:=\left\{f\in H^s(\R^2) : \forall i\in\{1,2,\dots,n\}: \int_{\R^2}\overline{\eta_i}\cdot f\:dx=\int_{\R^2}\overline{g_i}\cdot f\: dx\right\}
\end{equation*}
is dense in $L^2(\R^2)$ and $A\upharpoonright_{\mathfrak{D}}$ is a closed and densely defined restriction of $A$.

\end{example}

Let us finish this section by giving a restatement of Corollary \ref{coro:denserestrictions}. We will, however, omit the proof as it follows straightforwardly.
\begin{corollary} \label{coro:gelfanddenserestr}
Let $A^*$ be a closed and densely defined on a Hilbert space $\mathcal{H}$. Then there is a one-to-one correspondence between all densely defined and closed restrictions of $A^*$ and all subspaces $\mathcal{L}\subset\mathcal{H}_{-1}$ that are closed with respect to $\|\cdot\|_{-1}$ and that satisfy $\mathcal{L}\cap\mathcal{H}=\{0\}$. This correspondence is given via the operators $A_\mathcal{L}'$ as defined in \eqref{eq:al}.
\end{corollary}

\section{Convergence} \label{sec:convergence}
Given a sequence $\{A_n\}_{n=1}^\infty$ of closed operators on a Hilbert space $\mathcal{H}$, let us now recall Kato's notion of generalized convergence. Since we are only dealing with closed operators on Hilbert spaces, we will not define generalized convergence in its full generality which comprises operators between Banach spaces (cf.\ \cite[Chapt.\ IV,\S 2]{Kato} for the general definition). Rather, we will give a more special definition of generalized convergence that by {\cite[p.\ 198, footnote 1]{Kato}} is equivalent to the general definition in the Hilbert space case.
\begin{definition}\normalfont
Let $\{A_n\}_{n=1}^\infty$ and $A$ be closed operators on a Hilbert space $\mathcal{H}$. We say $\{A_n\}_{n=1}^\infty$ {\bf{converges in the generalized sense}} to a closed operator $A$, which we denote by ``$A_n\overset{\text{Kato}}{\longrightarrow} A$" if and only if the orthogonal projections onto the graphs $P(\Gamma(A_n))$ converge in norm to $P(\Gamma(A))$, i.e. if and only if
\begin{equation*}
\|P(\Gamma(A_n))-P(\Gamma(A))\|\overset{n\rightarrow\infty}{\longrightarrow}0\:.
\end{equation*}
\end{definition}
\begin{theorem} Let $A$ be a closed and densely defined operator on a complex Hilbert space $\mathcal{H}$. Let $\mathfrak{M}\subset\mathcal{D}(A^*)$ be a subspace and let $\{\mathfrak{M}_n\}_{n=1}^\infty$ be a sequence of subspaces of $\mathcal{D}(A^*)$. 
\\\\
i) Let the operators $C_{\mathfrak{M}_n}$ and $C_\mathfrak{M}$ be defined as in Definition \eqref{def:restr}. Moreover, for any linear space  $\mathfrak{N}\subset\mathcal{D}(A^*)$ define the set $$\Gamma_\mathfrak{N}=\{(\phi,A^*\phi):\phi\in\mathfrak{N}\}\:,$$
where $\Gamma_\mathfrak{N}\subset\mathcal{H}\oplus\mathcal{H}$.
Then $$C_{\mathfrak{M}_n}\overset{\text{Kato}}{\longrightarrow} C_\mathfrak{M}\quad\text{if and only if} \quad \lim_{n\rightarrow 0}\|P(\Gamma_{\mathfrak{M}_n})-P(\Gamma_\mathfrak{M})\|=0\:.$$ 
\\\\
ii) Assume in addition that the spaces $\{\mathfrak{M}_n\}_{n=1}^\infty$ and $\mathfrak{M}$ are closed with respect to $\|\cdot\|_{\Gamma(A^*)}$ and that they satisfy Condition \eqref{eq:assumptions}. Moreover, let the operators $A_{\mathfrak{M}_n}$ be defined as in Definition \ref{def:extension}. Then
$$A_{\mathfrak{M}_n}\overset{\text{Kato}}{\longrightarrow} A_\mathfrak{M}\quad\text{if and only if} \quad \lim_{n\rightarrow 0}\|P(\Gamma_{\mathfrak{M}_n})-P(\Gamma_\mathfrak{M})\|=0\:.$$
\label{thm:convergence}
\end{theorem},
\begin{proof} Observe that for any $\mathfrak{N}\subset\mathcal{D}(A^*)$ we have by \eqref{eq:perpi}
\begin{equation*}
P(\Gamma(C_\mathfrak{N}))=P(\Gamma(A^*))-P(\Gamma_\mathfrak{N})\:.
\end{equation*}
Thus, for any $n\in\N$, we get 
$$\|P(\Gamma(C_{\mathfrak{M}_n}))-P(\Gamma(C_{\mathfrak{M}}))\|=\|P(\Gamma_{\mathfrak{M}_n})-P(\Gamma_\mathfrak{M})\|\:,$$
from which Assertion i) immediately follows.

Next, let us show ii). By Lemma \ref{lemma:adjeq}, we know that $C_\mathfrak{M}=A_\mathfrak{M}^*$ and $C_{\mathfrak{M}_n}=A_{\mathfrak{M}_n}^*$ for any $n\in\N$. Moreover, since we assumed $\mathfrak{M}_n$ and $\mathfrak{M}$ to be closed with respect to $\|\cdot\|_{\Gamma(A^*)}$, we know by Lemma \ref{lemma:closable}, that $A_\mathfrak{M}$ and $\{A_{\mathfrak{M}_n}\}_{n=1}^\infty$ are closed operators. The result now follows from \cite[IV, Thm.\ 2.23, d)]{Kato}. 
\end{proof}
\begin{remark} \normalfont For the case that $\mathfrak{M}$ is one-dimensional, i.e.\ if $\mathfrak{M}=\spann\{\phi\}$, where $\|\phi\|_{\Gamma(A^*)}=1$, observe that $P(\Gamma_{\mathfrak{M}_n})$ converges to $P(\Gamma_\mathfrak{M})$ in norm if and only if for $n$ large enough we have that $\mathfrak{M}_n$ is one-dimensional and for each $n$, there exists a $\phi_n\in\mathfrak{M}_n$, where $\|\phi_n\|_{\Gamma(A^*)}=1$ such that $\|\phi_n-\phi\|_{\Gamma(A^*)}\overset{n\rightarrow\infty}{\longrightarrow}0\:.$ \label{rem:1-d}
\end{remark}
\begin{example} \normalfont \label{ex:converge}
Let $\mathcal{H}=L^2(\R)$ and for any $n\in\N$, consider the operator $C_n$ given by
\begin{align*}
C_n:\qquad\mathcal{D}(C_n)&=\left\{f\in H^1(\R): \sum_{j=0}^{n-1} f\left(\frac{j}{n}\right)=0\right\}\\
f&\mapsto if'\:.
\end{align*}
(Here and in the following, $f'$ denotes the weak derivative of $f$.)
We claim that $\{C_n\}_{n=1}^\infty$ is a sequence of densely defined operators that converges in the generalized sense to the non-densely defined operator $C_\infty$ given by
\begin{align*}
C_\infty:\qquad\mathcal{D}(C_\infty)&=\left\{f\in H^1(\R): \int_0^1 f(x)dx=0\right\}\\
f&\mapsto if'\:.
\end{align*}
To this end, let us firstly define the functions $\phi_{\lambda}(x):=\frac{1}{2}e^{-\left|x-\lambda\right|}$, where $\lambda\in\R$. Let $A=A^*$ be the selfadjoint momentum operator on the real axis:
\begin{align} \label{eq:momentum}
A:\qquad\mathcal{D}(A)=H^1(\R), \qquad f\mapsto if'\:.
\end{align}
It is not hard to see that for any $f\in\mathcal{D}(A)=H^1(\R)$ and for any $\lambda\in\R$, we get\footnote{Note that even though $A=A^*$, for the sake of clarity, we will keep the notation of Section \ref{sec:theory} in which the distinction between $A$ and $A^*$ was important.}
\begin{equation} \label{eq:pointcond}
\langle f,\phi_\lambda\rangle+\langle A^*f,A^*\phi_\lambda\rangle=\langle f,\phi\rangle+\langle f',\phi'\rangle=\overline{f(\lambda)}\:.
\end{equation}
Now, for any $n\in\N$, let us define the function $\psi_n(x):=\sum_{j=0}^{n-1}\phi_{\frac{j}{n}}(x)$ and the subspace $\mathfrak{M}_n:=\spann\{\psi_n\}$. Equation \eqref{eq:pointcond} then implies that $C_n=C_{\mathfrak{M}_n}$, where the operators $C_{\mathfrak{M}_n}$ are defined as in Definition \ref{def:restr}. 
Moreover, we have $\phi_{\lambda}(x)\in\mathcal{D}(A^*)=\mathcal{D}(A)$ as well as $A^*\phi_{\lambda}\notin\mathcal{D}(A)$ since $\phi_\lambda'(x)=-\frac{1}{2}\sgn(x-\lambda)e^{-|x-\lambda|}\notin H^1(\R)$, which immediately implies that $\psi_n\in\mathcal{D}(A^*)$ but $0\neq A^*\psi_n\notin\mathcal{D}(A)$. Theorem \ref{thm:denselydefined} therefore implies that the operators $C_n=C_{\mathfrak{M}_n}$ are densely defined.

In order to find $\psi_\infty\in\mathcal{D}(A^*)$ and $\mathfrak{M}_\infty:=\spann\{\psi_\infty\}$, such that $C_\infty=C_{\mathfrak{M}_\infty}$, we proceed with our analysis by applying the Fourier transform $\mathcal{F}$, where we use the convention
\begin{equation*}
\left(\mathcal{F}f\right)(k):=\frac{1}{\sqrt{2\pi}}\int_{-\infty}^\infty f(x)e^{-ikx}dx
\end{equation*}
for any $f\in L^1(\R)$. Then, the functions $\widehat{\phi}_\lambda(k):=(\mathcal{F}\phi_\lambda)(k)$ are given by
\begin{equation*}
\widehat{\phi}_\lambda(k)=\frac{1}{\sqrt{2\pi}}\frac{e^{-i\lambda k}}{1+k^2}\:.
\end{equation*}
Moreover, we have that $\widehat{\psi}_n(k)=(\mathcal{F}\psi_n)(k)$ is given by
\begin{equation*}
\widehat{\psi}_n(k)=\sum_{j=0}^{n-1}\widehat{\phi}_{\frac{j}{n}}(k)=\frac{1}{\sqrt{2\pi}}\sum_{j=0}^{n-1}\frac{e^{-i\frac{j}{n}k}}{1+k^2}=\frac{1}{\sqrt{2\pi}}\frac{(1-e^{-ik})}{(1-e^{-i\frac{k}{n}})(1+k^2)}\:,
\end{equation*}
from which we get in particular that 
\begin{equation*}
\|\widehat{\psi}_n\|^2_{\Gamma(\widehat{A}^*)}=\|\widehat{\psi}_n\|^2+\|k\widehat{\psi}_n\|^2=\frac{1}{2\pi}\int_{-\infty}^\infty\frac{1}{1+k^2}\cdot\frac{1-\cos(k)}{1-\cos\left(\frac{k}{n}\right)}dk\:,
\end{equation*}
where the diagonalized operator $\widehat{A}:=\mathcal{F}^*A\mathcal{F}$ is the selfadjoint maximal multiplication operator by the independent variable:
\begin{align*}
\widehat{A}:\qquad\mathcal{D}(\widehat{A})&=\left\{\widehat{f}\in L^2(\R): \int_{-\infty}^{\infty}k^2|\widehat{f}(k)|^2dk<\infty\right\}\\
\left(\widehat{A}\widehat{f}\right)(k)&=k\widehat{f}(k)\:.
\end{align*}
We now want to consider the normalized sequence $\{\widehat{\psi}_n/\|\widehat{\psi}_n\|_{\Gamma(\widehat{A}^*)}\}_{n=1}^\infty$ and find its limit $\widehat{\psi}_\infty$ in $\|\cdot\|_{\Gamma(\widehat{A}^*)}$-norm. By what has been said in Remark \ref{rem:1-d}, we then only need to apply inverse Fourier transform to find the function $\psi_\infty=\mathcal{F}^*\widehat{\psi}_\infty$, which then is the limit of the sequence $\{\psi_n/\|\psi_n\|_{\Gamma(A^*)}\}_{n=1}^\infty$ in $\|\cdot\|_{\Gamma(A^*)}$-norm. If we then show that $C_\infty=C_{\mathfrak{M}_\infty}$ this proves that $C_n\overset{\text{Kato}}{\longrightarrow}C_\infty$. 
To this end, observe that if $n\in\N$, then for almost every $k\in\R$,  we can estimate
\begin{equation*}
0\leq\frac{1-\cos(k)}{n^2(1-\cos\left(\frac{k}{n}\right))}\leq 1\:,
\end{equation*}
Thus, by dominated convergence we get
\begin{equation} \label{eq:normasympt}
\lim_{n\rightarrow\infty}\frac{1}{n^2}\|\widehat{\psi}_n\|_{\Gamma(\widehat{A}^*)}^2=\lim_{n\rightarrow\infty}\frac{1}{2\pi}\cdot\frac{1}{ n^2}\int_{-\infty}^\infty\frac{1}{1+k^2}\cdot\frac{1-\cos(k)}{1-\cos\left(\frac{k}{n}\right)}dk=\frac{1}{\pi}\int_{-\infty}^\infty\frac{1-\cos(k)}{k^2+k^4}dk=\frac{1}{e}\:,
\end{equation}
where the last integral can be computed using residual calculus. This shows that $\|\widehat{\psi}_n\|\sim e^{-1/2} n$ as $n\rightarrow\infty$. Thus, consider the normalized functions 
\begin{equation*}
\frac{\widehat{\psi}_n(k)}{\|\widehat{\psi}_n\|_{\Gamma(\widehat{A}^*)}}=\frac{1}{\sqrt{2\pi}}\frac{n}{\|\widehat{\psi}_n\|_{\Gamma(\widehat{A}^*)}}\left(\frac{1}{n}\sum_{j=0}^{n-1}e^{-i\frac{j}{n}k}\right)\frac{1}{1+k^2}\:.
\end{equation*}
We now claim that $\frac{\widehat{\psi}_n(k)}{\|\widehat{\psi}_n\|_{\Gamma(\widehat{A}^*)}}$ converges in $\|\cdot\|_{\Gamma(\widehat{A}^*)}$-norm to the function
\begin{equation*}
\widehat{\psi}_\infty(k)=\sqrt{\frac{e}{2\pi}}\cdot\frac{1}{1+k^2}\int_0^1 e^{-ijk}dj=\sqrt{\frac{e}{2\pi}}\cdot\frac{1}{1+k^2}\cdot\frac{1-e^{-ik}}{ik}\:.
\end{equation*}
Using dominated convergence again, this can be easily seen from the fact that by \eqref{eq:normasympt}, $\frac{n}{\|\widehat{\psi}_n\|_{\Gamma(\widehat{A}^*)}}$ converges to $\sqrt{e}$, while the Riemann--sums $\left(\frac{1}{n}\sum_{j=0}^{n-1}e^{-i\frac{j}{n}k}\right)$ converge to the integral $\int_0^1 e^{-ijk}dj$.
An application of the inverse Fourier transform then yields
\begin{equation*}
\psi_\infty(x):=\left(\mathcal{F}^*\widehat{\psi}_\infty\right)(x)=\sqrt{e}\left(\frac{\sgn(x-1)}{2}e^{-|x-1|}-\frac{\sgn(x)}{2}e^{-|x|}+\chi_{[0,1]}(x)\right)\:,
\end{equation*}   
where 
\begin{equation*}
\sgn(x):=\begin{cases} -1\quad&\text{if}\quad x<0\\ 1\quad&\text{if}\quad x\geq 0\end{cases}\quad\text{and}\quad\chi_{[0,1]}(x):=\begin{cases} 1\quad&\text{if}\quad x\in[0,1]\\ 0\quad&\text{else}\end{cases}\qquad\:.
\end{equation*} 
Observe that $\psi_\infty\in H^2(\R)=\mathcal{D}(AA^*)=\mathcal{D}(AA)$, which can be easiest seen from the fact that its Fourier transform $\widehat{\psi}_\infty$ is in the domain of $\widehat{A}\widehat{A}$ -- the maximal operator of multiplication by $k^2$:
\begin{equation*}
\int_{-\infty}^\infty |k^2\widehat{\psi}_\infty(k)|^2dk=\frac{e}{\pi}\int_{-\infty}^\infty\left(\frac{k}{1+k^2}\right)^2(1-\cos(k))\,dk\leq 2\int_{-\infty}^\infty\left(\frac{k}{1+k^2}\right)^2dk<\infty\:.
\end{equation*}
A calculation shows that
\begin{equation*}
\psi_{\infty}(x)+\left(AA^*\psi_\infty\right)(x)=\psi_\infty(x)-\psi_\infty''(x)=\sqrt{e}\cdot\chi_{[0,1]}(x)\:.
\end{equation*}
Hence, defining $\mathfrak{M}_\infty:=\spann\{\psi_\infty\}$, then implies that $\mathcal{D}(C_\infty)=\mathcal{D}(C_{\mathfrak{M}_\infty})$ and in particular that $C_n\overset{\text{Kato}}{\longrightarrow}C_\infty$. Moreover, even though it is obvious, note that Theorem \ref{thm:denselydefined} and the fact that $\psi_\infty\in H^2(\R)$ imply that $C_{\mathfrak{M}_\infty}=C_\infty$ cannot be densely defined. 
\end{example}
\section{Closable and non-closable extensions} \label{sec:clonoclo}
As in the previous section, let $A$ be the selfadjoint momentum operator on the real line (cf. \ \eqref{eq:momentum}). We will now construct two extensions $A_\mathfrak{M}$ of $A$, where $\mathfrak{M}$ has infinite dimension and where for any $0\neq\phi\in\mathfrak{M}$ we have that $A^*\phi\notin\mathcal{D}(A)$, but only one of them describes a closable extension of $A$, respectively a densely defined restriction of $A^*$. As in the previous section, for any $\lambda\in\R$, let us define the function $\phi_\lambda(x):=\frac{1}{2}e^{-|x-\lambda|}$.
Now, define
\begin{align*}
\mathfrak{M}_\Z&=\spann\{\phi_\lambda: \lambda\in \Z\}\\
\mathfrak{M}_\Q&=\spann\{\phi_\lambda: \lambda\in \Q\}
\end{align*}
i.e.\ the set of finite linear combinations of vectors $\phi_\lambda$, where $\lambda\in \Z, \Q$.
Recall  that $(A^*\phi_\lambda)(x)=-\frac{1}{2}\sgn(x-\lambda)e^{-|x-\lambda|}\notin 
\mathcal{D}(A)=H^1(\R)$ for any $\lambda\in\R$, implying that both $\mathfrak{M}_\Z$ and $\mathfrak{M}_\Q$ satisfy Condition \eqref{eq:precloscon}.
Now, let us show that $\mathfrak{M}_\R:=\spann\{\phi_\lambda: \lambda\in \R\}\subset\overline{\mathfrak{M}_\Q}^{\|\cdot\|_{\Gamma(A^*)}}$. To see this let $\lambda\in\R$ and just pick any sequence $\{\lambda_n\}_{n=1}^\infty\subset \Q$ such that $\lambda_n\rightarrow \lambda$ and consider
\begin{align*}
&\lim_{n\rightarrow\infty}\|\phi_{\lambda}-\phi_{\lambda_n}\|^2_{\Gamma(A^*)}\\=&\lim_{n\rightarrow\infty}\left(\left\|\frac{1}{2}e^{-|x-\lambda|}-\frac{1}{2}e^{-|x-\lambda_n|}\right\|^2+\left\|\frac{1}{2}\sgn(x-\lambda)e^{-|x-\lambda|}-\frac{1}{2}\sgn(x-\lambda_n)e^{-|x-\lambda_n|}\right\|^2\right)=0\:,
\end{align*}
which follows from dominated convergence. Next, let us determine $\mathfrak{M}_\R^{\perp_{\Gamma(A^*)}}$, i.e. we want to determine all $f\in\mathcal{D}(A^*)=H^1(\R)$ such that
\begin{equation} \label{eq:perpend}
\langle f,\phi\rangle+\langle f',\phi'\rangle=0\quad\text{for all}\:\: \phi\in\mathfrak{M}_\R\:,
\end{equation}
 and hence in particular for all $\phi_\lambda$, where $\lambda\in\R$. In \eqref{eq:pointcond}, we have already shown that Condition \eqref{eq:perpend} implies
$$0=\langle f,\phi_\lambda\rangle+\langle f',\phi_\lambda'\rangle=\overline{f(\lambda)}\quad\text{for all}\:\:\lambda\in\R\:,$$
meaning that $$\left(\overline{\mathfrak{M}_\Q}^{\|\cdot\|_{\Gamma(A^*)}}\right)^{\perp_{\Gamma(A^*)}}\subset\mathfrak{M}_\R^{\perp_{\Gamma(A^*)}}=\{0\}\:.$$ From this, we get that $\overline{\mathfrak{M}_\Q}^{\|\cdot\|_{\Gamma(A^*)}}=L^2(\R)$. Hence, $$\overline{\mathfrak{M}_\Q}^{\|\cdot\|_{\Gamma(A^*)}}\cap\mathcal{D}(A)=L^2(\R)\cap H^1(\R)=H^1(\R)\neq\{0\}\:,$$ which means that $\mathfrak{M}_\Q$ does not satisfy the assumptions of Theorem \ref{thm:denselydefined} and thus the operator $A_{\mathfrak{M}_\Q}$ is not closable (in fact, its adjoint is the zero operator on the trivial space $\{0\}$).\\
Let us now consider the case $\mathfrak{M}_\Z$. Again by using \eqref{eq:pointcond}, it is not difficult to see that $$\mathfrak{M}_\Z^{\perp_{\Gamma(A^*)}}=\{f\in H^1(\R): f(z)=0\quad\text{for all}\:\: z\in\Z\}\:.$$ 
In particular, we have $\mathcal{C}_c^\infty(\R\setminus\Z)\subset \mathfrak{M}_\Z^{\perp_{\Gamma(A^*)}}$ and thus, any function $g\in\overline{\mathfrak{M}_\Z}^{\|\cdot\|_{\Gamma(A^*)}}=\left(\mathfrak{M}_\Z^{\perp_{\Gamma(A^*)}}\right)^{\perp_{\Gamma(A^*)}}$ has to satisfy
\begin{equation} \label{eq:pointagain}
\langle f,g\rangle+\langle f',g'\rangle=0\:,
\end{equation}
for any $f\in\mathcal{C}_c^\infty(\R\setminus\Z)$.
Now, assume that the condition $\{A^*\phi:\phi\in\overline{\mathfrak{M}_\Z}^{\|\cdot\|_{\Gamma(A^*)}}\}\cap\mathcal{D}(A)=\{0\}$ is not satisfied, i.e. that there exists a $\widetilde{g}\in\mathcal{D}(A)$ such that $\widetilde{g}'\in\overline{\mathfrak{M}_\Z}^{\|\cdot\|_{\Gamma(A^*)}}\cap\mathcal{D}(A)$ (observe that $\ker A=\{0\}$). Then, we could perform another integration by parts in \eqref{eq:pointagain} and would obtain
$$\langle f,\widetilde{g}\rangle+\langle f',\widetilde{g}'\rangle=\langle f,\widetilde{g}-\widetilde{g}''\rangle=0\quad\text{for any}\:\: f\in\mathcal{C}_c^\infty(\R\setminus\Z)\:.$$ 
However, this implies that $\widetilde{g}-\widetilde{g}''=0$ since $f$ is an arbitrary element of the dense set $\mathcal{C}_c^\infty(\R\setminus\Z)$. Moreover, since there is no $L^2(\R)$-solution to the equation $\widetilde{g}=\widetilde{g}''$, we get that $\widetilde{g}=0$. Thus, by Lemma \ref{lemma:closable}, the operator $A_{\mathfrak{M}_\Z}$ is closable.

\section{Restrictions of selfadjoint operators} \label{sec:restrictions}

Given an unbounded selfadjoint reference operator $S=S^*$ on $\mathcal{H}$, let us construct densely defined and closed and thus in particular symmetric restrictions $C_\phi\subset S$ such that $\dim(\mathcal{D}(S)\setminus\mathcal{D}(C_\phi))=1$. We will restrict ourselves to the one-dimensional case for simplicity of presentation, however our results can be generalized straightforwardly.

For any $0\neq\phi\in\mathcal{D}(S)$ such that $S\phi\notin\mathcal{D}(S)$ let us define the operator $C_\phi$ as follows:
\begin{equation*}
C_\phi:\quad\mathcal{D}(C_\phi)+\{f\in\mathcal{D}(S):\langle \phi,f\rangle+\langle S\phi,Sf\rangle=0\},\qquad C_\phi=S\upharpoonright_{\mathcal{D}(C_\phi)}\:.
\end{equation*}
Noting that $S=S^*$ this means that $C_\phi=C_{\mathfrak{M}}(S)$, where $\mathfrak{M}=\spann\{\phi\}$ (cf. Definition \ref{def:restr}). By Theorem \ref{thm:denselydefined}, the condition $S\phi\notin\mathcal{D}(S)$ ensures that $C_\phi$ is densely defined. Moreover, since $C_\phi$ is a restriction of the selfadjoint operator $S$, this implies in particular that $C_\phi$ is symmetric and thus $C_\phi\subset C_\phi^*$. By Lemma \ref{lemma:adjeq}, $C_\phi^*$ is given by
\begin{align*}
C_\phi^*:\qquad\mathcal{D}(C_\phi^*)&=\mathcal{D}(S)\dot{+}\spann\{S\phi\}\\
f+\lambda S\phi&\mapsto Sf-\lambda \phi\:.
\end{align*}

In order to determine all selfadjoint extensions of $C_\phi$, let us firstly compute the defect spaces $\ker(C_\phi^*\mp i)$:
\begin{align*}
0&=(C_\phi^*\mp i)(f+\lambda S\phi)=(S\mp i)f+\lambda (-\phi\mp iS\phi)\Leftrightarrow
f=\lambda(S\mp i)^{-1}(\phi\pm iS\phi)\:,
\end{align*}
which implies that 
\begin{equation} \label{eq:defspace}
\ker(C_\phi^*\mp i)=\spann\{(S\mp i)^{-1}(\phi\pm iS\phi)+S\phi\}=\spann\{{(S\pm i)\phi}\}\:.
\end{equation}
By von Neumann's Theorem (e.g. {\cite[Satz 10.9]{Weid1}}), we know that all selfadjoint extensions of $C_\phi$ can be parametrized by unitary maps from $\ker(C_\phi^*-i)$ to $\ker(C_\phi^*+i)$.
Since $\|(S-i)\phi\|=\|(S+i)\phi\|$, they are given by
\begin{align} \label{eq:sadjext}
C_{\phi,\vartheta}:\qquad\qquad\mathcal{D}(C_{\phi,\vartheta})=\mathcal{D}(C_\phi)\dot{+}\spann\{(S+ i)\phi&+e^{i\vartheta}(S-i)\phi\}\notag\\
f+\lambda((S+i)\phi+e^{i\vartheta}(S-i)\phi)&\mapsto Sf+i\lambda((S+i)\phi-e^{i\vartheta}(S-i)\phi)\:,
\end{align}
where $\vartheta\in(-\pi,\pi]$. Note that, independently of the choice of $\phi$, we have $C_{\phi,\pi}=S$. This follows from the fact that $\phi\notin\mathcal{D}(C_\phi)$ but $\mathcal{D}(C_{\phi,\pi})=\mathcal{D}(C_\phi)\dot{+}\spann\{\phi\}\subset\mathcal{D}(S)$, from which we get equality by a dimension counting argument. 

Without loss of generality, assume that $\|(S+i)\phi\|=\|(S-i)\phi\|=1$ from now on. Let us now determine the resolvents $(C_{\phi,\vartheta}+ i)$ of the extensions $C_{\phi,\vartheta}$, which have to coincide on $\ran(C_{\phi}+i)=\spann\{(S+i)\phi\}^\perp$. Moreover, since we have
$$ (C_{\phi,\vartheta}+i)\left[(S+i)\phi+e^{i\vartheta}(S-i)\phi\right]=2i(S+i)\phi\in\ker(C_{\phi}^*-i)\:,$$ we get that
\begin{equation} \label{eq:resref}
(C_{\phi,\vartheta}+i)^{-1}(S+i)\phi=\frac{1}{2i}\left[(S+i)\phi+e^{i\vartheta}(S-i)\phi\right]\:.
\end{equation}
 Hence, since $(C_{\phi,\vartheta}+i)^{-1}\upharpoonright_{\ran(C_\phi+i)}=(S+i)^{-1}\upharpoonright_{\ran(C_\phi+i)}$ and by \eqref{eq:resref}, we get 
 $$\left[(C_{\phi,\vartheta}+i)^{-1}-(S+i)^{-1}\right](S+i)\phi=\frac{1+e^{i\vartheta}}{2i}(S-i)\phi\:,$$
 which implies that --- as an identity of operators --- we have
 \begin{equation} \label{eq:resolventdifference}
 (C_{\phi,\vartheta}+i)^{-1}=(S+i)^{-1}+\frac{1+e^{i\vartheta}}{2i}|(S-i)\phi\rangle\langle(S+i)\phi|\:.
 \end{equation}
Here, ``$|(S-i)\phi\rangle\langle(S+i)\phi|$" denotes the rank--one operator which maps $\psi\mapsto \langle(S+i)\phi,\psi\rangle (S-i)\phi$.

\begin{example}\normalfont \label{ex:laplacian}
Let $\Omega$ be an open, bounded domain in $\R^n$ (with $n\geq 2$) such that $\partial\Omega$ is smooth. Now consider the selfadjoint Dirichlet Laplacian $\Delta_D$ on $L^2(\Omega)$ given by
\begin{equation*}
\Delta_{D}\:\:\:\: :\quad\mathcal{D}(\Delta_{D})=\left\{f\in H^2(\Omega): f\upharpoonright_{\partial\Omega}=0\right\},\quad f\mapsto \Delta f \:.
\end{equation*}
Let $g\in L^2(\Omega)$, $h\in\mathcal{C}(\partial\Omega)$ and consider the following restriction of $\Delta_D$:
\begin{align} \label{eq:laprec}
\widetilde{\Delta}_{D}:\qquad\mathcal{D}(\widetilde{\Delta}_{D})=\left\{f\in\mathcal{D}(\Delta_D): \int_\Omega \overline{g}\cdot f\,dx=\int_{\partial\Omega} \overline{h}\cdot \partial_\nu f\,d\sigma\right\},\quad f\mapsto\Delta f\:
\end{align} 
where $\partial_\nu f$ denotes the normal derivative of $f$.
Now, let $$\phi(g,h):=(\Delta_D+i)^{-1}(\Delta_D-i)^{-1}g-\Delta_D(\Delta_D+i)^{-1}(\Delta_D-i)^{-1}\eta_h\:,$$
where $\eta_h\in\mathcal{C}(\overline{\Omega})$ is the unique function harmonic on $\Omega$ with $\eta_h\upharpoonright_{\partial\Omega}=h$. Note that $\phi(g,h)\in\mathcal{D}(\Delta_D)$. A calculation -- using Green's identity -- then shows that for any $f\in\mathcal{D}(\Delta_D)$, we have

\begin{equation*}
\langle \phi(g,h),f\rangle+\langle \Delta_D\phi(g,h),\Delta_Df\rangle=0\:\Leftrightarrow    \int_\Omega \overline{g}\cdot f\,dx=\int_{\partial\Omega} \overline{h}\cdot \partial_\nu f\,d\sigma\:,
\end{equation*}
which implies that 
\begin{equation*}
\mathcal{D}(\widetilde{\Delta}_D)=\{f\in\mathcal{D}(\Delta_D):\langle \phi(g,h),f\rangle+\langle \Delta_D\phi(g,h),\Delta_Df\rangle=0\}\:.
\end{equation*}
In other words, $\widetilde{\Delta}_D=C_{\mathfrak{M}}(\Delta_D)$, with $\mathfrak{M}=\spann\{\phi(g,h)\}$. By Theorem \ref{thm:denselydefined}, $\mathcal{D}(\widetilde{\Delta}_D)$ is dense in $L^2(\Omega)$ if and only if $\Delta_D\phi(g,h)\notin\mathcal{D}(\Delta_D)$. Since
\begin{equation} \label{eq:domadd}
\Delta_D\phi(g,h)=\underbrace{\Delta_D(\Delta_D+i)^{-1}(\Delta_D-i)^{-1}g+(\Delta_D+i)^{-1}(\Delta_D-i)^{-1}\eta_h}_{\in\mathcal{D}(\Delta_D)}-\eta_h\:,
\end{equation}
this is the case if and only if $h\neq 0 \:\Leftrightarrow\: \eta_h\neq 0$, which we will assume from now on. Note that this shows denseness of $\mathcal{D}(\widetilde{\Delta}_D)$ in $L^2(\Omega)$ just by making operator-algebraic manipulations when applying $\Delta_D$ to $\phi(g,h)$. For later convenience, we may assume without loss of generality that
$\int_{\partial\Omega}|h|^2d\sigma=1$ at this point. Let us now also introduce the minimal and maximal realizations of the Laplacian on $\Omega$, $\Delta_{min}$ and $\Delta_{max}$ respectively:

\begin{align*}
\Delta_{min}&:\quad\mathcal{D}(\Delta_{min})=\left\{f\in H^2(\Omega): f\upharpoonright_{\partial\Omega}={\partial_\nu f}\upharpoonright_{\partial\Omega}=0\right\},\quad f\mapsto \Delta f \:,\\
\Delta_{max}&:\quad\mathcal{D}(\Delta_{max})=\{f\in L^2(\Omega): \Delta f\in L^2(\Omega)\},\qquad\qquad\:\: f\mapsto \Delta f \:.\\
\end{align*}
It is well-known that $\Delta_{min}^*=\Delta_{max}$ and that $\Delta_D$ is the Friedrichs extension of $\Delta_{min}$ from which by \cite[Lemma 2.5]{Alonso-Simon} we get
\begin{equation} \label{eq:adjdomain}
\mathcal{D}(\Delta_{max})=\mathcal{D}(\Delta_D)\dot{+}\ker(\Delta_{max})\:,
\end{equation}
where $\ker(\Delta_{max})$ turns out to be the space of functions in $L^2(\Omega)$ which are harmonic on $\Omega$. In particular, since for any $h\in\mathcal{C}(\partial\Omega)$, we have $\eta_h\in\ker(\Delta_{max})$,  we get from \eqref{eq:domadd} and \eqref{eq:adjdomain} that $\Delta_D\phi(g,h)\in\mathcal{D}(\Delta_{max})$.
Moreover, from \eqref{eq:laprec}, we see that $\Delta_{min}\subset\widetilde{\Delta}_D$ if and only if $g\equiv 0$. On the other hand, by Lemma \ref{lemma:adjeq}
\begin{align*}
\mathcal{D}(\widetilde{\Delta}_D^*)&=\mathcal{D}(\Delta_D)\dot{+}\spann\{\Delta_D\phi(g,h)\}=\mathcal{D}(\Delta_D)\dot{+}\spann\{\eta_h\}\subset\mathcal{D}(\Delta_{max})\:,
\end{align*}
i.e.\ ``boundary conditions" as in \eqref{eq:laprec} describe restrictions $\widetilde{\Delta}_D$ of $\Delta_D$ whose adjoint $\widetilde{\Delta}_D^*$ has domain contained in $\mathcal{D}(\Delta_{max})$. Another calculation  -- using \eqref{eq:domadd} -- now shows that $\widetilde{\Delta}_D^*\eta_h=g$, from which we get
\begin{align*}
\widetilde{\Delta}_D^*:\qquad \mathcal{D}(\widetilde{\Delta}_D^*)=\mathcal{D}(\Delta_D)\dot{+}\spann\{\eta_h\},\qquad (f_D+\lambda\eta_h)\mapsto \Delta f_D+\lambda g
\end{align*}
for any $f_D \in\mathcal{D}(\Delta_D)$ and any $\lambda\in\C$. Using that for any $f\in\mathcal{D}(\widetilde{\Delta}_D^*)$, we have that $f\upharpoonright_{\partial\Omega}\in\spann\{h\}$, it is not hard to see that $\widetilde{\Delta}_D^*$ can also be characterized as follows:
\begin{align*}
\widetilde{\Delta}_D^*:\quad\mathcal{D}(\widetilde{\Delta}_D^*)=\left\{f\in H^2(\Omega): f\upharpoonright_{\partial\Omega}=h\int_{\partial\Omega}\overline{h}fd\sigma\right\},\quad f\mapsto \Delta f+g\int_{\partial\Omega}\overline{h}fd\sigma\:.
\end{align*}
Let $n_\pm$ now denote the defect elements of $\widetilde{\Delta}_D$, i.e. 
$\ker(\widetilde{\Delta}_D^*\mp i)=\spann\{n_\pm\}$.
By \eqref{eq:defspace}, we get 
\begin{equation} \label{eq:defel}
n_\pm=(\Delta_D\pm i)\phi=(\Delta_D\mp i)^{-1}g-\Delta_D(\Delta_D\mp i)^{-1}\eta_h=(\Delta_D\mp i)^{-1}(g\mp i\eta_h)-\eta_h    
\end{equation}
and therefore all selfadjoint extensions of $\widetilde{\Delta}_D$ are given by
\begin{align*}
\widetilde{\Delta}_{D,\vartheta}:\quad\mathcal{D}(\widetilde{\Delta}_{D,\vartheta})=\mathcal{D}(\widetilde{\Delta}_{D})\dot{+}\spann\{n_++e^{i\vartheta}n_-\},\qquad \widetilde{\Delta}_{D,\vartheta}=\widetilde{\Delta}_D^*\upharpoonright_{\mathcal{D}(\widetilde{\Delta}_{D,\vartheta})}\:.
\end{align*}
As argued above, the extension $\widetilde{\Delta}_{D,\pi}$ is equal to the selfadjoint Dirichlet Laplacian $\Delta_D$. Hence, let us only consider $\vartheta\in(-\pi,\pi)$ from now on.
In order to describe all the other selfadjoint extensions of $\widetilde{\Delta}_D$ in terms of boundary conditions, firstly observe that
\begin{align} \label{eq:defidty}
C:=\:\:\:\:\:&\int_{\partial\Omega}\overline{h}\partial_\nu n_+d\sigma-\langle g,n_+\rangle=-\int_{\partial\Omega}\overline{n_-}\partial_\nu n_+d\sigma-\langle g,n_+\rangle\notag\\=-&\int_{\partial\Omega}\overline{\partial_\nu n_-}\cdot n_+d\sigma-\int_\Omega\left(\overline{n_-}\Delta n_+-n_+\overline{\Delta n_-}\right)dx-\langle g,n_+\rangle\notag\\=-&\int_{\partial\Omega}\overline{\partial_\nu n_-} \cdot n_+d\sigma-\int_\Omega\left(\overline{n_-}(in_++g)-n_+\overline{(-in_-+g)}\right)dx-\langle g,n_+\rangle\notag\\=\:\:\:\:\:&\int_{\partial\Omega}{h}\overline{\partial_\nu n_-}d\sigma-\langle n_-,g\rangle\:,
\end{align}
where we have used that by \eqref{eq:defel}, $n_\pm\upharpoonright_{\partial\Omega}=-h$ and $\Delta n_\pm=\pm i n_\pm+g$. On the other hand, by virtue of the same identities, we get
\begin{align} \label{eq:ImC}
C&=\int_{\partial\Omega}\overline{h}\partial_\nu n_+d\sigma-\langle g,n_+\rangle=-\int_{\partial\Omega}\overline{n_+}\partial_\nu n_+d\sigma-i\|n_+\|^2-\langle\Delta n_+,n_+\rangle\notag\\&=\|\nabla n_+\|^2-\int_{\partial\Omega}\left(\overline{n_+}\partial_\nu n_++n_+\overline{\partial_\nu n_+}\right)d\sigma                                                                                                                                                                                                                                                                                                                                                                                                                                                                                                                                        -i\|n_+\|^2\:,
\end{align}
which shows that $\Imag(C)=-\|n_+\|^2\neq 0$. Now, let $f=\widetilde{f}_D+\lambda(n_++e^{i\vartheta}n_-)$ be an arbitrary element of $\mathcal{D}(\widetilde{\Delta}_{D,\vartheta})$, where $\widetilde{f}_D\in\mathcal{D}(\widetilde{\Delta}_D)$ and $\lambda\in\C$. We then get
\begin{equation} \label{eq:boundeval}
\int_{\partial\Omega}\overline{h}fd\sigma=\lambda\int_{\partial\Omega}\overline{h}(n_++e^{i\vartheta}n_-)d\sigma=-\lambda(1+e^{i\vartheta})\int_{\partial\Omega}|h|^2d\sigma=-\lambda(1+e^{i\vartheta})
\end{equation}
as well as 
\begin{align*} \label{eq:boundeval2}
\int_{\partial\Omega}\overline{h}\partial_\nu fd\sigma-\langle g,f\rangle&=\lambda\left[\int_{\partial\Omega}\overline{h}\partial_\nu n_+d\sigma-\langle g,n_+\rangle+e^{i\vartheta}\left(\int_{\partial\Omega}\overline{h}\partial_\nu n_-d\sigma-\langle g,n_-\rangle\right)\right]\notag\\&\overset{\eqref{eq:defidty}}{=}\lambda(C+e^{i\vartheta}\overline{C})\overset{\eqref{eq:boundeval}}{=}-\frac{C+e^{i\vartheta}\overline{C}}{1+e^{i\vartheta}}\int_{\partial\Omega}\overline{h}fd\sigma=(-\Real(C)-\tan(\vartheta/2)\cdot\Imag(C))\int_{\partial\Omega}\overline{h}fd\sigma\notag\\&=r(\vartheta)\cdot\int_{\partial\Omega}\overline{h}fd\sigma\:.
\end{align*}
Now, since by \eqref{eq:ImC}, $\Imag(C)\neq 0$, the map $r:\vartheta\mapsto (-\Real(C)-\tan(\vartheta/2)\cdot\Imag(C))$ as a bijection from $(-\pi,\pi)$ to $\R$. From this, it is easy to see that all selfadjoint extensions of $\widetilde{\Delta}_D$ can also be described as
\begin{equation*}
\widetilde{\Delta}_{D,r}:\quad\mathcal{D}(\widetilde{\Delta}_{D,r})=\left\{f\in\mathcal{D}(\widetilde{\Delta}_D^*): \int_{\partial\Omega} \overline{h}\partial_\nu fd\sigma-\langle g,f\rangle=r\int_{\partial\Omega}\overline{h}fd\sigma \right\},\quad \widetilde{\Delta}_{D,r}=\widetilde{\Delta}_D^*\upharpoonright_{\mathcal{D}(\widetilde{\Delta}_{D,r})}\;,
\end{equation*}
where $r\in\R\cup\{\infty\}$ with the convention that $\widetilde{\Delta}_{D,\infty}$ corresponds to the Dirichlet Laplacian $\Delta_D$.
\end{example}
\subsection*{Acknowledgements} I am very grateful to Sergii Kuzhel for valuable feedback on this manuscript and for providing useful references.  Parts of this work have been done during my PhD studies at the University of Kent in Canterbury, UK (cf.\ \cite[Chapter 4]{thesis}), and I would also like to thank my thesis advisors Sergey Naboko and Ian Wood for support and guidance. Finally, I would like to express my appreciation to the UK Engineering and Physical Sciences Research Council (Doctoral Training Grant Ref.\ EP/K50306X/1) and the School of Mathematics, Statistics and Actuarial Science at the University of Kent for a PhD studentship.

\end{document}